\documentclass[11pt,reqno]{amsart}
\usepackage{amsmath,amssymb,bbm,subfig,url}

\usepackage{tikz}
\usetikzlibrary{calc}
\usetikzlibrary{arrows}
\usetikzlibrary{patterns}
\usetikzlibrary{through}
\usetikzlibrary{plotmarks}

\newtheorem{theorem}{Theorem}[section]
\newtheorem{lemma}[theorem]{Lemma}

\newtheorem{corollary}[theorem]{Corollary}
\newtheorem{conjecture}[theorem]{Conjecture}

\newtheorem{problem}[theorem]{Problem}

\numberwithin{equation}{section}

\newcommand{\R}{{\mathbb R}}

\begin{document}

\title[]{Maximizing Neumann fundamental tones of triangles}

\author[]{R. S. Laugesen and B. A. Siudeja}
\address{Department of Mathematics, University of Illinois, Urbana,
IL 61801, U.S.A.} \email{Laugesen\@@illinois.edu}
\email{Siudeja\@@illinois.edu}
\date{\today}

\keywords{Isoperimetric, free membrane.}
\subjclass[2000]{\text{Primary 35P15. Secondary 35J20}}

\begin{abstract}
We prove sharp isoperimetric inequalities for Neumann eig\-envalues
of the Laplacian on triangular domains.

The first nonzero Neumann eigenvalue is shown to be maximal for
the equilateral triangle among all triangles of given perimeter,
and hence among all triangles of given area. Similar results are
proved for the harmonic and arithmetic means of the first two
nonzero eigenvalues.
\end{abstract}

\maketitle

\vspace*{-12pt}

\section{\bf Introduction}

Eigenvalues of the Laplacian arise in physical models of wave motion, diffusion (such as heat flow) and quantum mechanics, namely as frequencies, rates of decay and energy levels. The eigenvalues are constrained by geometric considerations. For example, writing $\mu_1$ for the first nonzero eigenvalue of the Laplacian under Neumann boundary conditions on a domain of area $A$ in the plane, one has that
\[
\text{$\mu_1 A$ is maximal for disks.}
\]
This result is due to Szeg\H{o} \cite{S54} for simply connected domains, with the extension to all domains and all dimensions by Weinberger \cite{W56}. 

Thus a free membrane of given area has highest fundamental tone when the membrane is circular, and the temperature of an insulated region of given volume will relax most quickly to equilibrium when the region is spherical. 

We prove a sharper result for triangular domains in the plane:
\[
\text{$\mu_1 A$ is maximal when the triangle is equilateral.}
\]
It remains open to extend this result to $n$-gons ($n \geq 4$), and to find a higher dimensional result involving tetrahedra.

Our result on $\mu_1 A$ for triangles generalizes in three different
ways: to a stronger geometric functional, to a stronger eigenvalue
functional, and to a trade-off between the two. To strengthen the geometric functional, we write $L$ for the perimeter and prove
\[
\text{$\mu_1 L^2$ is maximal for the equilateral triangle,}
\]
which implies the result for $\mu_1 A$ by invoking the
triangular isoperimetric inequality.

Strengthening instead the eigenvalue functional, we show
\[
\Big( \frac{1}{\mu_1} + \frac{1}{\mu_2} \Big)^{\! \! -1} A
\text{\quad is maximal in the equilateral case.}
\]
That is, our result on the fundamental tone extends to the harmonic mean
of the first two nonzero eigenvalues.

We trade off a further strengthening of the eigenvalue functional
against a weakening of the geometric functional. Specifically, we
show the arithmetic mean $(\mu_1 + \mu_2)/2$ of the first two
non-zero eigenvalues is maximal for the equilateral triangle, after
normalizing the ratio of the square of the area to the sum of the squares
of the side lengths.

Our primary method is Rayleigh's Principle and the method of trial
functions. Linearly transplanted eigenfunctions of the equilateral
triangle are used to handle triangles that are close to equilateral, and
linear or quadratic trial functions handle all the others. Neither the
conformal mapping approach of Szeg\H{o} nor the ``radial
extension'' method of Weinberger seems to work for triangles.

Our triangle results suggest new open problems for general
domains, such as a possible strengthening of the
Szeg\H{o}--Weinberger bound by an isoperimetric excess term, as
explained in Section~\ref{discussion}.

Our companion paper \cite{LS09b} \emph{minimizes} $\mu_1$ among
triangles, under a diameter normalization, with the minimizer being
the degenerate acute isosceles triangle. We know of no other
papers in the literature that study sharp isoperimetric type
inequalities for Neumann eigenvalues of triangles. The Neumann eigenfunctions of triangles were investigated for the ``hot spots'' conjecture, by Ba\~{n}uelos and Burdzy \cite{BaBu}.

Dirichlet eigenvalues of triangles have received considerable
attention \cite{AF06,AF08,F06,fresiu,LR08,S07,S09}, as
discussed in Section~\ref{dirich}. Dirichlet eigenvalues of
degenerate domains have also been investigated lately
\cite{BF09,F07}.

For a modern perspective on the Szeg\H{o}--Weinberger result, including its role as a prototype for
Payne--P\'{o}lya--Weinberger type inequalities, see
the survey paper by Ashbaugh \cite{A99}. Generalizations of the Szeg\H{o}--Weinberger result to closed surfaces such as the Klein bottle, the sphere, genus $2$ surface, projective plane and equilateral torus are known too \cite{EGJ06,H70,JLNNP05,LY82,N96}. For broad surveys of
isoperimetric eigenvalue inequalities, one can consult the
monographs of Bandle \cite{B80}, Henrot \cite{He06}, Kesavan
\cite{K06} and P\'{o}lya--Szeg\H{o} \cite{PS51}.

\section{\bf Notation}
\label{notation}

The Neumann eigenfunctions of the Laplacian on a bounded plane
domain $\Omega$ with Lipschitz boundary satisfy $-\Delta u = \mu u$ with
natural boundary condition $\partial u / \partial n = 0$. The
eigenvalues $\mu_j$ are nonnegative, with
\[
0 = \mu_0 < \mu_1 \leq \mu_2 \leq \dots \to \infty .
\]
Call $\mu_1$ the \textbf{fundamental tone}, since
$\sqrt{\mu_1}$ is proportional to the lowest frequency of vibration
of a free membrane over the domain. Call the eigenfunction $u_1$ a
\textbf{fundamental mode}. The \textbf{Rayleigh Principle} says
\[
\mu_1 = \min_{v \perp 1} R[v]
\]
where $v \perp 1$ means $\int_\Omega v \, dA = 0$, and where the Rayleigh quotient is
\[
R[v] = \frac{\int_\Omega |\nabla v|^2 \, dA}{\int_\Omega v^2 \, dA} \qquad \text{for\ } v \in H^1(\Omega) .
\]

For triangular domains, we write:
\begin{itemize}
  \item[] $A=$ area,
  \item[] $l_1 \ge l_2 \ge l_3 > 0$ for the lengths of the sides,
  \item[] $L=l_1 + l_2 + l_3 =$ perimeter,
  \item[] $S^2=l_1^2+l_2^2+l_3^2 =$ sum of squares of side lengths.
\end{itemize}
We will not need this next fact, but it is interesting that $S^2 = 36I/A$ where $I$ is the moment of inertia of the triangular region \cite[formula (6)]{FLL07}.

Given nonnegative numbers $a_j$, define their
\begin{align*}
  \text{arithmetic mean} & = M(a_1,\ldots,a_n) = \frac{a_1 + \cdots + a_n}{n} , \\
  \text{harmonic mean} & = H(a_1,\ldots,a_n) = 1 \Big/ \Big( \frac{1/a_1 + \cdots + 1/a_n}{n} \Big) .
\end{align*}

Denote the first positive roots of the Bessel functions $J_0, J_1, J_1^\prime$, by
\[
j_{0,1} \simeq 2.4048 , \qquad
j_{1,1} \simeq 3.8317 , \qquad
j_{1,1}^\prime \simeq 1.8412 .
\]
%
%\begin{align*}
%j_{0,1} & \simeq 2.4048 , \\
%j_{1,1} & \simeq 3.8317 , \\
%j_{1,1}^\prime & \simeq 1.8412 .
%\end{align*}
%

\section{\bf Isoperimetric upper bounds on the fundamental tone}

The Szeg\H{o}--Weinberger result says that among
all domains of given volume, the first nonzero Neumann
eigenvalue of the Laplacian is maximized by a ball. Thus in two
dimensions, 
\begin{equation} \label{SW1}
\mu_1 A \leq \pi (j_{1,1}^\prime)^2 .
\end{equation}
Our first theorem proves a stronger inequality for triangles.
\begin{theorem}\label{1upS}
For all triangles,
\begin{gather} \label{SWbetter1}
    \mu_1 S^2\leq \frac{16\pi^2}{3}
\end{gather}
and hence
\begin{gather} \label{SWbetter2}
    \mu_1 L^2\leq 16\pi^2
\end{gather}
and
\begin{gather} \label{SWbetter3}
    \mu_1 A\leq \frac{4\pi^2}{3\sqrt{3}} .
\end{gather}
In each inequality, equality holds if and only if the triangle is
equilateral.
\end{theorem}
Inequality \eqref{SWbetter3} for triangles improves significantly on
Szeg\H{o} and Weinberger's estimate \eqref{SW1}, because
$4\pi^2/3\sqrt{3} \simeq 7.6$ is much less than $\pi
(j_{1,1}^\prime)^2 \simeq 10.7$.

The implications \eqref{SWbetter1} $\Rightarrow$ \eqref{SWbetter2}
$\Rightarrow$ \eqref{SWbetter3} are immediate from the following
geometric inequalities.
\begin{lemma} \label{averages}
For all triangles,
\[
  12\sqrt{3}A \le L^2 \le 3S^2 .
\]
In each inequality, equality holds if and only if the triangle is
equilateral.
\end{lemma}
The left hand inequality is simply the triangular isoperimetric
inequality. It implies that the \emph{triangular isoperimetric
excess}
\begin{equation} \label{eq:triexcess}
  {\mathcal E}_T=\frac{L^2}{12\sqrt{3}}-A
\end{equation}
is nonnegative and equals $0$ only for equilateral triangles.

\begin{theorem}\label{1opt}
For all triangles,
\begin{gather*}
  \mu_1 \cdot \left(A+\frac{\pi^2}{j_{0,1}^2}{\mathcal E}_T\right) \leq \frac{4\pi^2}{3\sqrt{3}}
\end{gather*}
with equality only for equilateral triangles. Equality also holds
asymptotically for degenerate obtuse isosceles triangles.
\end{theorem}
The discussion in Section~\ref{discussion} motivates such bounds
involving the isoperimetric excess, and shows that
Theorem~\ref{1opt} implies Theorem~\ref{1upS}.

So far we have maximized the fundamental tone under normalizations of the area, perimeter and sum of squares of the side lengths. If instead one normalizes the longest side, which equals the diameter of the triangle, then the optimal result is known already: for all convex plane domains of diameter $D$, 
\[
\mu_1 D^2 < 4j_{0,1}^2
\]
by work of Cheng \cite[Theorem~2.1]{cheng}. This estimate saturates for degenerate obtuse isosceles triangles, as discussed for example in our companion paper \cite[Proposition 3.6]{LS09b}.

One can bound the harmonic mean of the first \textit{two} nonzero
eigenvalues. For simply connected domains in two dimensions, the optimal inequality under area normalization is 
\begin{equation} \label{SW2}
H(\mu_1,\mu_2) A \leq \pi (j_{1,1}^\prime)^2
\end{equation}
with equality for disks, by work of Szeg\H{o} and Weinberger \cite[p.~634]{W56}. (For non-simply connected domains, the best result to date is $H(\mu_1,\mu_2) A \leq 4\pi$ by Ashbaugh and Benguria \cite{AB93}.) For triangles we have a stronger result:
\begin{theorem}\label{12upA}
For all triangles,
\[
    H(\mu_1,\mu_2)A \leq \frac{4\pi^2}{3\sqrt{3}}
\]
with equality if and only if the triangle is equilateral.
\end{theorem}
An even stronger inequality is conjectured in Section~\ref{discussion}, using perimeter.

Obviously Theorem~\ref{12upA} for the harmonic mean implies inequality
\eqref{SWbetter3} for the first eigenvalue.

Next, we strengthen the harmonic mean of the eigenvalues to the
arithmetic mean, at the cost of weakening the geometric functional
from $A$ to $A^2/S^2$.
\begin{theorem}\label{12upAS}
For all triangles,
\begin{gather*}
    M(\mu_1,\mu_2)\frac{A^2}{S^2}\leq \frac{\pi^2}{9}
\end{gather*}
with equality if and only if the triangle is equilateral.
\end{theorem}
By multiplying the inequalities in Theorems~\ref{12upA} and
\ref{12upAS}, we obtain an estimate on the geometric mean of the
first two nonzero eigenvalues.
\begin{corollary} \label{cor:geom}
For all triangles,
\[
\mu_1 \mu_2 \frac{A^3}{S^2} \leq \frac{4\pi^4}{27\sqrt{3}}
\]
with equality if and only if the triangle is equilateral.
\end{corollary}
A stronger inequality is conjectured in Section~\ref{discussion}, using just the area.

\section{\bf Eigenfunctions of the equilateral triangle} \label{equilateral}

This section gathers together the first three Neumann eigenfunctions
and eigenvalues of the equilateral triangle, which we use later to
construct trial functions for close-to-equilateral triangles.

\subsection{The equilateral triangle} \label{sec:equi}
The modes and frequencies of the equilateral triangle were derived
two centuries ago by Lam\'{e}, albeit without a proof of completeness. We present the first few modes below. For proofs, see the recent exposition (including completeness) by McCartin
\cite{M02}, building on work of Pr\'{a}ger \cite{Pr98}. A different approach is due to Pinsky \cite{Pi80}.

Consider the the equilateral triangle $E$ with vertices at $(0,0)$,
$(1,0)$ and $(1/2,\sqrt{3}/2)$. Then $\mu_0=0$, with eigenfunction
$u_0 \equiv 1$, and
\[
\mu_1 = \mu_2 = \frac{16 \pi^2}{9}
\]
with eigenfunctions
  \begin{align*}
    u_1(x,y) & = 2 \Big[ \cos \big( \frac{\pi}{3}(2x-1) \big) + \cos \big( \frac{2\pi y}{\sqrt{3}} \big) \Big] \sin \big( \frac{\pi}{3} (2x-1)
    \big) , \\
    u_2(x,y) & = \cos \big( \frac{2\pi}3(2x-1) \big) - 2\cos\big( \frac\pi3(2x-1) \big) \cos\big( \frac{2\pi y}{\sqrt{3}} \big) .
  \end{align*}
Clearly $u_1$ is antisymmetric with respect to the line of symmetry
$\{ x=1/2 \}$ of the equilateral triangle, since
$u_1(1-x,y)=-u_1(x,y)$, whereas $u_2$ is symmetric with respect to
that line.

It is easy to check that equality holds for
the equilateral triangle in Theorems~\ref{1upS}, \ref{1opt}, \ref{12upA} and \ref{12upAS},
because $\mu_1 = 16 \pi^2/9$ and $S^2=3, L=3$ and $A=\sqrt{3}/4$.

We evaluate some integrals of $u_1$ and $u_2$, for later use:
\begin{align*}
\int_E u_1^2 \, dA & = \int_E u_2^2 \, dA = \frac{3\sqrt{3}}{8} \\
\int_E \Big( \frac{\partial u_1}{\partial x} \Big)^{\! 2} \, dA & = \int_E \Big( \frac{\partial u_2}{\partial y} \Big)^{\! 2} \, dA = \frac{32\pi^2 + 243}{32\sqrt{3}} \\
\int_E \Big( \frac{\partial u_1}{\partial y} \Big)^{\! 2} \, dA & = \int_E \Big( \frac{\partial u_2}{\partial x} \Big)^{\! 2} \, dA = \frac{32\pi^2 - 243}{32\sqrt{3}} \\
\int_E \frac{\partial u_1}{\partial x} \frac{\partial u_1}{\partial y} \, dA & = \int_E \frac{\partial u_2}{\partial x} \frac{\partial u_2}{\partial y} \, dA = 0
\end{align*}
and also some integrals of cross-terms:
\begin{align*}
\int_E u_1 u_2 \, dA = \int_E \frac{\partial u_1}{\partial x} \frac{\partial u_2}{\partial x} \, dA & = \int_E \frac{\partial u_1}{\partial y} \frac{\partial u_2}{\partial y} \, dA = 0 \\
\int_E \frac{\partial u_1}{\partial x} \frac{\partial u_2}{\partial y} \, dA & = \frac{81\sqrt{3}}{32}+\pi \\
\int_E \frac{\partial u_1}{\partial y} \frac{\partial u_2}{\partial x} \, dA & = \frac{81\sqrt{3}}{32}-\pi
\end{align*}

\subsection{Transplanting the eigenfunctions}

Here we transplant functions from the equilateral triangle $E$ to an arbitrary triangle $T$. Assume $T$ has vertices at $(-1,0), (1,0)$ and $(a,b)$, where $b>0$. Write
\[
q = a^2 + b^2 + 3 .
\]
Let $\tau$ be the affine transformation of $E$ to $T$ that maps the vertices $(0,0)$,
$(1,0),(1/2,\sqrt{3}/2)$ to $(-1,0),(1,0),(a,b)$, respectively. Its inverse is
\[
\tau^{-1}(x,y)=\big( (1+x-ay/b)/2 , \sqrt{3}y/2b \big) .
\]
Given a function $u$ on $E$, define $v = u \circ \tau^{-1}$ on $T$ by
\[
v(x,y) = (u \circ \tau^{-1})(x,y) = u \big( (1+x-ay/b)/2 , \sqrt{3}y/2b \big) .
\]
If $u$ has mean value zero, $\int_E u \, dA = 0$, then so does $v$, with $\int_T v \, dA = 0$. By straightforward changes of variable,
\begin{equation} \label{eq:ray1}
\frac{\int_T |\nabla v|^2 \, dA}{\int_T v^2 \, dA} = \frac{\int_E \big[ (a^2 + b^2)u_x^2 - 2\sqrt{3}a u_x u_y + 3u_y^2 \big] \, dA}{4b^2 \int_E u^2 \, dA} .
\end{equation}
In particular, taking a linear combination $u=\gamma u_1 + \delta u_2$ of the eigenfunctions on $E$, we let $v = u \circ \tau^{-1}$ to deduce
\begin{equation} \label{eq:ray2}
\frac{\int_T |\nabla v|^2 \, dA}{\int_T v^2 \, dA} = \frac{[(32\pi^2 + 243)q-1458]\gamma^2 - 972 \sqrt{3} a \gamma \delta + [(32\pi^2 - 243)q+1458] \delta^2}{144b^2 (\gamma^2 + \delta^2)}
\end{equation}
by substituting $u=\gamma u_1 + \delta u_2$ into \eqref{eq:ray1} and recalling the integrals in Section~\ref{sec:equi}.

Similarly, putting $v_1 = u_1 \circ \tau^{-1}$ and $v_2=u_2 \circ \tau^{-1}$ implies
\begin{equation} \label{eq:ray3}
\frac{\int_T \nabla v_1 \cdot \nabla v_2 \, dA}{\int_T |\nabla v_1|^2 \, dA} = - \frac{486 \sqrt{3} a}{(32\pi^2 + 243)q-1458} ,
\end{equation}
by changing variable back to $E$ and then using integrals from Section~\ref{sec:equi}.

\section{\bf Proof of Theorem \ref{1upS} and Lemma \ref{averages}}

First we prove Lemma \ref{averages}. Recall that $M(a_1,\ldots,a_n)$
denotes the arithmetic mean and define
\begin{itemize}
\item[] $G(a_1,\ldots,a_n) = \sqrt[n]{a_1 \cdots a_n} =$ geometric
  mean,
\item[] $Q(a_1,\ldots,a_n) = \sqrt{(a_1^2 + \cdots + a_n^2)/n} =$ quadratic
  mean.
\end{itemize}
Then for any triangle,
\begin{align*}
12\sqrt{3}A
& = 3\sqrt{3L} \, G(L-2l_1,L-2l_2,L-2l_3)^{3/2} && \text{by Heron's formula} \\
& \le 3\sqrt{3L} \, M(L-2l_1,L-2l_2,L-2l_3)^{3/2} \\
& = L^2 \\
& = 9M(l_1,l_2,l_3)^2 \\
& \le 9Q(l_1,l_2,l_3)^2 \\
& = 3S^2.
\end{align*}
Inequalities between these means become equalities if and only if
$l_1=l_2=l_3$, meaning the triangle is equilateral. Thus Lemma
\ref{averages} is proved.

\vspace{6pt} For Theorem \ref{1upS}, it remains to prove $\mu_1 S^2
\leq 16\pi^2/3$ with equality if and only if the triangle is
equilateral.

Let $T$ be a triangle. By rescaling and rotating and reflecting, we can assume the longest side of $T$ has length $2$ with vertices at $(-1,0)$ and $(1,0)$,
and that the third vertex $(a,b)$ satisfies $a \geq 0$ and $b>0$ and
\begin{equation}
    (a+1)^2+b^2 \leq 2^2 . \label{cond4}
\end{equation}
Let us express these constraints in terms of new variables
\[
p=b^2 \qquad \text{and} \qquad q=3+a^2+b^2 .
\]
By definition, $q > 3$ and $0 < p \leq q-3$. Condition \eqref{cond4}
says $2a \leq 6-q$, and since $a \geq 0$ we conclude $q \leq 6$.
Hence
\begin{equation} \label{eq:constr1}
3 < q \leq 6 , \qquad 0 < p \leq q-3 .
\end{equation}
Further, by substituting $a=\sqrt{q-3-p}$ into $2a \leq 6-q$ and
then squaring, we find
\begin{equation} \label{eqQc}
p \geq Q_c(q) = -\frac{1}{4}q^2+4q-12 .
\end{equation}
The constraint region determined by \eqref{eq:constr1}--\eqref{eqQc}
is plotted in Figure~\ref{fig:constraint}.

\begin{figure}[t]
    \begin{center}
  \begin{tikzpicture}[xscale=3,yscale=1.5]
    \draw[->] (3,0) -- (6.2,0) node [below=-2pt] {\small $q$};
    \draw[->] (3,0) -- (3,3.2) node [left=-2pt] {\small $p$};
    \draw (3,0) +(0,0.03) -- +(0,-0.03) node [below] {\tiny $3$};
    \draw (4,0) +(0,0.03) -- +(0,-0.03) node [below] {\tiny $4$};
    \draw (6,0) +(0,0.03) -- +(0,-0.03) node [below] {\tiny $6$};
    \draw (3,0) +(0.015,0) -- +(-0.015,0) node [left] {\tiny $0$};
    \draw (3,3) +(0.015,0) -- +(-0.015,0) node [left] {\tiny $3$};
    \draw[dotted] (3,3) -| (6,0);
    \draw[domain=4:6,fill=black!40] plot (\x,-\x*\x/4+4*\x-12) --
(3,0) -- (4,0);
  \end{tikzpicture}
    \end{center}
    \caption{The constraint region in the proof of Theorem \ref{1upS}, showing the line $p=q-3$ and the
curve $p=Q_c(q)$.}
    \label{fig:constraint}
  \end{figure}
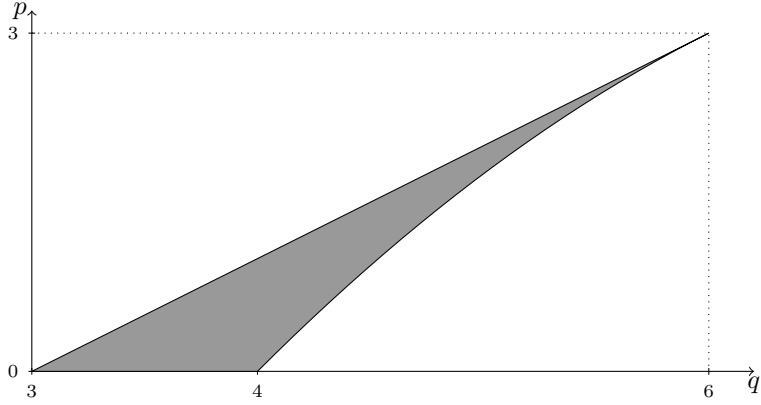

When $q=6$ the constraints require $p=3$, so that $a=0,b=\sqrt{3}$,
and so $T$ is equilateral. In that case equality holds in the
theorem, with $\mu_1 S^2 = 16\pi^2/3$ by Section~\ref{equilateral}.
So from now on we assume $q<6$.

Continuing with the proof, note the squares of the side lengths of
the triangle add up to
\begin{equation} \label{Sis}
S^2 = 2^2 + (a-1)^2 + b^2 + (a+1)^2 + b^2 = 2q.
\end{equation}

Consider now the linear functions
\[
    f(x,y)=x-\frac{a}{3}, \qquad g(x,y)=y-\frac{b}{3} ,
\]
which integrate to zero over the triangle $T$. Our first trial function is the linear
combination $v=f+\gamma g$ where $\gamma \in \R$. By Rayleigh's Principle,
\begin{align}
\mu_1 S^2 & \leq R[f+\gamma g] S^2 \notag \\
& = \frac{18(1+\gamma^2)}{3+(a+\gamma b)^2} \, 2q . \label{eq:trial1}
\end{align}
This last expression is less than $16\pi^2/3$ (as desired for the
theorem) if
\[
    27(1+\gamma^2)q-4\pi^2 \big( 3+(a+\gamma b)^2 \big)<0.
\]
The left hand side is a quadratic polynomial in $\gamma$, and hence
an appropriate $\gamma$ exists if the discriminant is positive,
which is equivalent to
\begin{equation} \label{region1}
    p < Q_\ell(q) = \frac{9}{16\pi^4} (4\pi^2-27)q^2 .
\end{equation}

For our second trial function, let $u_1$ be the first antisymmetric
eigenfunction of the equilateral triangle $E$ and recall the affine transformation $\tau$ from $E$ onto $T$, as described in Section~\ref{equilateral}. The transplanted function $v=u_1 \circ \tau^{-1}$ integrates to 0 over $T$, and so can be used as a trial function. By Rayleigh's Principle,
\begin{align*}
\mu_1 S^2 & \leq R[u_1 \circ \tau^{-1}] S^2 \\
& = \frac{(32\pi^2+243)q-1458}{144p} \, 2q ,
\end{align*}
where the Rayleigh quotient has been evaluated by formula \eqref{eq:ray2} with $\gamma=1$ and $\delta=0$. Notice the last expression is less than $16\pi^2/3$ if
\begin{equation}
    p > Q_1(q) = \frac{1}{384\pi^2} \big( (32\pi^2+243)q^2-1458q \big) . \label{region2}
\end{equation}

Observe $Q_c(q)>Q_1(q)$ when
\[
5.03 \simeq \frac{768 \pi^2}{128 \pi^2 + 243} < q < 6 .
\]
Therefore if $5.04 \leq q < 6$ then the constraint \eqref{eqQc} implies
$p \geq Q_c(q) > Q_1(q)$, so that \eqref{region2} holds and hence $\mu_1 S^2 < 16\pi^2/3$.

Next, observe $Q_\ell(q)>Q_1(q)$ when
\[
0 < q < \frac{1458\pi^2}{32 \pi^4 - 621\pi^2 + 5832} \simeq 5.10 .
\]
Therefore if $3 < q \leq 5.09$ then for all $p \in \R$, either $p <
Q_\ell(q)$ or else $p \geq Q_\ell(q) > Q_1(q)$, so that either
\eqref{region1} or \eqref{region2} holds; in either case, we conclude $\mu_1 S^2 < 16\pi^2/3$.

We have proved $\mu_1 S^2 < 16\pi^2/3$ in the whole constraint region $3< q < 6$, and so the proof is
complete.

\vspace{6pt} To summarize the above proof, notice that
close-to-equilateral triangles (with $5.04 \leq q < 6$ above) are handled by the trial function $u_1$, which is the linearly transplanted eigenfunction of the equilateral triangle, while
far-from-equilateral triangles ($3 < q \leq 5.09$) are treated with either that same transplanted eigenfunction or else the linear trial function $f+\gamma g$.

\section{\bf Proof of Theorem \ref{1opt}} \label{pr:1opt}
Equality in the theorem holds for equilateral triangles, as observed in Section~\ref{equilateral}. In addition, when an obtuse isosceles triangle degenerates towards a line segment, equality holds in the limit because $A \to 0, L \to 2D$ and $\mu_1 \to 4j_{0,1}^2 D^{-2}$ (by \cite[Proposition~3.6]{LS09b}).

Assume for the rest of the proof that the triangle is non-equilateral. By rescaling, rotating and reflecting, we reduce to considering the triangle $T$ with vertices $(-1,0), (1,0)$ and $(a,b)$, where $a \geq 0, b>0$ and all the sidelengths are less than or equal to $2$. This triangle has area $A=b$ and diameter $D=2$.

Introduce new parameters $r=(l_2+l_3)/2$ and $s=(l_2-l_3)/2$ defined in terms of the sidelengths
\[
l_2 = \sqrt{(a+1)^2+b^2} \quad \text{and} \quad l_3 = \sqrt{(a-1)^2+b^2} .
\]
These new parameters occupy a triangular region in the $rs$-plane (see Figure~\ref{fig1} below) with
\[
1 < r \leq 2, \qquad 0 \leq  s < 1 , \qquad r+s \leq 2 ,
\]
since $l_3 \leq l_2 \leq l_1 = 2$ and $l_1 < l_2 + l_3$ and $l_2 < l_1 + l_3$. The line $r=1$ corresponds to degenerate triangles. Since $T$ is not equilateral, we know $(r,s) \neq (2,0)$ .

In terms of the new parameters, we have
\begin{align*}
a & = rs , \\
b^2 & = (r^2-1)(1-s^2) , \\
q & = a^2 + b^2 + 3 = r^2+s^2+2 , \\
L & = l_1 + l_2 + l_3 = 2(1+r) , \\
A & = b = \sqrt{(r^2-1)(1-s^2)} , \\
{\mathcal E}_T & = \frac{L^2}{12\sqrt{3}} - A = \frac{(1+r)^2}{3\sqrt{3}} - \sqrt{(r^2-1)(1-s^2)} .
\end{align*}

To handle close-to-degenerate triangles, we recall Cheng's bound $\mu_1 D^2 < 4j_{0,1}^2$ for convex domains (see \cite[Theorem~2.1]{cheng}); since our triangle $T$ has diameter $D=2$, Cheng's bound gives $\mu_1 < j_{0,1}^2$ and thus
\[
\mu_1 \Big( A+\frac{\pi^2}{j_{0,1}^2} {\mathcal E}_T \Big) < j_{0,1}^2 A + \pi^2 {\mathcal E}_T .
\]
This last expression is less than $4\pi^2/3\sqrt{3}$ (as desired for the theorem) if
\begin{equation} \label{sector}
(r-1)(r+3) - \Big( 1 - \frac{j_{0,1}^2}{\pi^2} \Big) \sqrt{27(r^2 - 1)(1-s^2)} <  0.
\end{equation}
The left side is increasing with $s$. Putting $s=2-r$ (the largest value of $s$ in our parameter region), we find that \eqref{sector} holds if
\[
(r+3) - \Big( 1 - \frac{j_{0,1}^2}{\pi^2} \Big) \sqrt{27(r+1)(3-r)} <  0.
\]
The expression on the left is convex for $-1 < r < 3$, and is negative at $r=1$ and $r=5/4$, and so  inequality \eqref{sector} certainly holds for $1 < r \leq 5/4$. Thus the theorem is proved in that range, as indicated in Figure~\ref{fig1}. Next, suppose $5/4 \leq r \leq 3/2$ and $0 \leq s \leq 1/3$. Replacing $s$ by $1/3$ in \eqref{sector}, and replacing $r+3$ by $(3/2)+3=9/2$, we see that it suffices to prove
\[
(r-1) \Big( \frac{9}{2} \Big)^{\! 2} - \Big( 1 - \frac{j_{0,1}^2}{\pi^2} \Big)^{\! 2} \cdot 27(r+1)\big( 1-(1/3)^2 \big) < 0.
\]
This linear inequality is easily established when $r \leq 3/2$. Hence the theorem holds when $5/4 \leq r \leq 3/2$ and $0 \leq s \leq 1/3$, as indicated in Figure~\ref{fig1}.

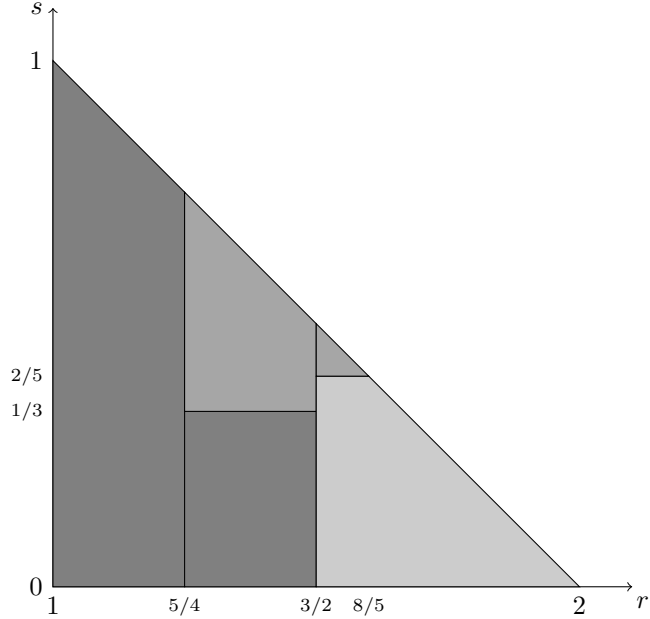
\begin{figure}[t]
\begin{center}
\begin{tikzpicture}[scale=7,smooth]
  \draw[<->] (0,1.1) -- (0,0) -- (1.1,0);
  \draw (0,0) node [below] {\small $1$}
        (1/4,0) node [below] {\tiny $5/4$}
        (1/2,0) node [below] {\tiny $3/2$}
        (3/5,0) node [below] {\tiny $8/5$}
        (1.12,0) node [below] {\small $r$}
        (1,0) node [below] {\small $2$}
        (0,0) node [left] {\small $0$}
        (0,1/3) node [left] {\tiny $1/3$}
        (0,2/5) node [left] {\tiny $2/5$}
        (0,1.1) node [left] {\small $s$}
        (0,1) node [left] {\small $1$};
  \fill [fill=black!50,draw=black] (0,0) -- (0,1) -- (1/4,3/4) -- (1/4,0) -- cycle;
  \fill [fill=black!50,draw=black] (1/4,0) rectangle (1/2, 1/3);
  \fill [fill=black!35,draw=black] (1/4,1/3) -- (1/4,3/4) -- (1/2,1/2) -- (1/2,1/3) --cycle;
  \fill [fill=black!35,draw=black] (1/2,2/5) -- (1/2,1/2) -- (3/5,2/5) -- cycle;
  \fill [fill=black!20,draw=black] (1/2,0) -- (1/2,2/5) -- (3/5,2/5) -- (1,0) -- cycle;
\end{tikzpicture}

\end{center}
\caption{Different shadings represent different trial functions used to prove Theorem~\ref{1opt}.
The darkest shading represents the circular sector-based trial function underlying Cheng's bound; intermediate shading represents a linear trial function; and the lightest shading represents a trial function based on eigenfunctions of the equilateral triangle.
The point $(2,0)$ corresponds to the equilateral triangle, and degenerate triangles have $r=1$.} \label{fig1}
\end{figure}

For the remaining parameter regions in Figure~\ref{fig1}, we will show
\begin{equation} \label{eq:opt1}
  R[v] \Big( A+\frac{12}{7} {\mathcal E}_T \Big) < \frac{4\pi^2}{3\sqrt{3}}
\end{equation}
for some trial function $v$ having mean value zero; notice here we have replaced $\pi^2/j_{0,1}^2$ with the slightly larger value $12/7$.

First, consider the linear trial function $v(x,y)=x-a/3$, which has mean value zero and $R[v]=18/(3+a^2)$. The desired estimate \eqref{eq:opt1} is then equivalent to
\begin{equation} \label{linearrs}
  (r+1)^2 - \frac{7\pi^2}{54}(3+r^2 s^2) < \frac{5}{12}\sqrt{27(r^2 - 1)(1-s^2)} .
\end{equation}
Restrict to the region where $5/4 \leq r \leq 3/2$ and $1/3 \leq s \leq 2-r$.
The left and right sides of inequality \eqref{linearrs} each decrease as $s$ increases, and so we put $s=2-r$
in the right side and $s=1/3$ in the left, reducing our task to proving the inequality
\[
(r+1)^2 - \frac{7\pi^2}{486} (27+r^2) - \frac{5}{12} \sqrt{27(r^2 - 1)(r-1)(3-r)} <  0 .
\]
The expression is convex for $1 < r < 3$, and so it is enough to verify the inequality at the endpoints $r=5/4$ and $r=3/2$. Direct calculation shows it is true at those endpoints, and so the theorem is proved when $5/4 \leq r \leq 3/2$ and $1/3 \leq s \leq 2-r$.

Next restrict to the region $3/2\le r\le 8/5$ and $2/5\le s\le 2-r$. Like above, we put $s=2-r$ in the right side and $s=2/5$ in the left. The resulting expression is again convex and the inequality is true at the endpoints $r=3/2$ and $r=8/5$.

The third case uses a stronger version of \eqref{eq:opt1} that is easier to handle. We have
\[
A+\frac{12}{7} {\mathcal E}_T=\frac{L^2}{7\sqrt{3}}-\frac{5}{7}A \leq \frac{L^2}{7\sqrt{3}}-\frac{180 A^2}{7\sqrt{3}L^2}
\]
by the triangular isoperimetric inequality $12\sqrt{3}A \leq L^2$. Thus \eqref{eq:opt1} will hold if we can show
\begin{equation} \label{eq:opt2}
R[v] (L^2 - 180\frac{A^2}{L^2}) < \frac{28\pi^2}{3}.
\end{equation}

Consider the eigenfunctions $u_1$ and $u_2$ of the equilateral triangle $E$, as in Section~\ref{equilateral}, and the affine transformation $\tau$ from $E$ to $T$. Transplant the eigenfunctions by $v_1=u_1 \circ \tau^{-1}$ and $v_2 = u_2 \circ \tau^{-1}$, so that $v_1$ and $v_2$ integrate to $0$ over $T$. Then take the trial function
\[
v = v_1 + \frac{1}{3} v_2 .
\]
Its Rayleigh quotient $R[v]$ can be evaluated by formula \eqref{eq:ray2} with $\gamma=1$ and $\delta=1/3$. Substituting in this formula for $R[v]$ reduces the desired estimate \eqref{eq:opt2} to
\begin{equation} \label{equirs}
  U(r,s) V(r,s) + W(r,s) < 0
\end{equation}
where
\begin{align*}
U(r,s) & = \left(\frac{320\pi^2}{9}+216 \right)(r^2+s^2)-324\sqrt{3}rs+\frac{640\pi^2}{9}-864 , \\
V(r,s) & = 4(r+1)^3 + 45(r-1)(s^2-1) , \\
W(r,s) & = \frac{4480\pi^2}{3}(r-1)(r+1)^2(s^2-1) .
\end{align*}
%
%One can easily check that $V \geq 0$ for $r \geq 1$. Also, $U \geq 0$ 
%since it is obtained from $R[v]$. Note that $W<0$ since $s<1$.

We will prove inequality \eqref{equirs} in the trapezoidal region
defined by $3/2 \leq r \leq 2$  and $0 \leq s \leq \min\left\{
2/5,2-r \right\}$. First we show $UV+W$ is convex with respect to $s$. The second derivative of $UV+W$ with respect to $s$ is
\begin{align*}
  432(32-132x+114x^2+49x^3)+\frac{640\pi^2}{9}(32+306x+282x^2+91x^3)\\
  -87480x(x+1)\sqrt{3}s+480(40\pi^2+243)xs^2 ,
\end{align*}
where we have put $x=r-1>0$. We want to show this quantity is positive. We may discard the $s^2$ term, since it is certainly positive. Further, we may replace $\sqrt{3}s$ by the larger number $1$, noting $s \leq 2/5$ in the trapezoidal region, and we may replace $\pi^2$ by the smaller number $9$. After these reductions, one is left with the polynomial
\[
34304 + 51336 x + 142248 x^2 + 79408 x^3 ,
\]
which is positive. Hence $UV+W$ is convex with respect to $s$, in the trapezoidal region.

By convexity, it suffices to prove \eqref{equirs} on the upper and lower boundary portions of the trapezoid, that is, where $3/2 \leq r < 2$ and $s=0$, or $3/2 \leq r\leq 8/5$ and $s=2/5$, or $8/5 \leq r < 2$ and $s=2-r$.

We start with $3/2 \leq r < 2$ and $s=0$. By substituting $s=0$ and $r=x+3/2$ into \eqref{equirs}, we reduce to the inequality
\begin{align*}
(x - 1/2) \Big( (160\pi^2 + 972) x^4+(1760\pi^2+10692) x^3 + (4680\pi^2+32805) x^2 & \\
+(3400\pi^2+35235) x + (-3100\pi^2+34020) \Big) & < 0,
\end{align*}
which obviously holds true for $0 \leq x < 1/2$, that is, for $3/2 \leq r < 2$. Incidentally, the root at $x=1/2$ arises from the equilateral triangle ($r=2$).

Now we take $8/5 \leq r < 2$ and $s=2-r$. By substituting $s=2-r$ and $r=x+8/5$ into \eqref{equirs}, we reduce to the inequality
\begin{align*}
(x - 2/5) \Big( (7000000 \pi^2 + 7441875 c) x^4 + (21400000 \pi^2 + 
    15278625 c) x^3 & \\
    + (34200000 \pi^2 + 
    8693325 c) x^2 + (7976000 \pi^2 + 
    12510855 c) x & \\ 
    + (-10211200 \pi^2 + 
   11572632 c) \Big) & < 0 ,
\end{align*}
where $c=4 + 3 \sqrt{3}$. Note the constant term in the quartic is positive. Thus the inequality holds when $0 \leq x < 2/5$, that is, when $8/5 \leq r < 2$. 

Lastly we take $3/2 \leq r \leq 8/5$ and $s=2/5$. By substituting $s=2/5$ and $r=x+3/2$ into \eqref{equirs}, we reduce to the inequality
\begin{align*}
(243000 + 40000 \pi^2) x^5 + (2551500 - 145800 \sqrt{3} + 
    420000 \pi^2) x^4 & \\
    + (7341030 - 1312200 \sqrt{3} + 
    1095600 \pi^2) x^3 + (6530625 - 2996190 \sqrt{3} & \\ 
    + 934600 \pi^2) x^2 + (4352859 - 3623130 \sqrt{3} - 
    138480 \pi^2) x & \\
    + (-4211433 - 2383830 \sqrt{3} + 820260 \pi^2) & < 0 .
\end{align*}
The coefficients of the second and higher powers of $x$ are positive, and so this quintic polynomial is convex. The polynomial is negative at $x=0$ and $x=1/10$, and hence is negative whenever $0 \leq x \leq 1/10$, that is, whenever $3/2 \leq r \leq 8/5$. This observation completes the proof.

\section{\bf Proof of Theorem \ref{12upA}}
The harmonic mean of the first two nonzero eigenvalues is characterized in terms of the Rayleigh quotient by Poincare's variational principle \cite[p.~99]{B80}:
\[
   H(\mu_1,\mu_2) = \min \big\{ H(R[f_1],R[f_2]) : \text{$f_1 \perp 1, f_2 \perp 1$ and $\nabla f_1 \perp \nabla f_2$} \big\}.
\]

As in the earlier proofs, we reduce to considering the triangle $T$ with vertices
$(-1,0), (1,0)$ and $(a,b)$ where $a \geq 0, b>0$ and $(a+1)^2+b^2 \leq 2^2$. Notice $a < 1$ and $b \leq \sqrt{3}$, and that the triangle has area $A=b$. Recall the definition
\[
q = a^2 + b^2 + 3
\]
and observe
\begin{equation} \label{eq:qb}
q \geq b^2 + 3 \geq 2 \sqrt{3} b .
\end{equation}

Consider the polynomial trial functions
\begin{align*}
  f_1(x,y)&=x-\frac{a}{3},\\
  f_2(x,y)&=\Big( x - \frac{a}{3} \Big)^{\! 2} - \frac{3+a^2}{18},
\end{align*}
whose coefficients have been chosen to ensure $f_1$ and $f_2$ have mean value zero over $T$ and have orthogonal gradients ($\int_T \nabla f_1 \cdot \nabla f_2 \, dA = 0$). These trial functions have Rayleigh quotients
\[
R[f_1] = \frac{18}{3+a^2} , \qquad R[f_2] = \frac{360}{7(3+a^2)} .
\]
Hence by Poincar\'{e}'s principle,
\begin{align*}
H(\mu_1,\mu_2)A & \le H(R[f_1],R[f_2])A \\
& =\frac{80b}{3(3+a^2)} .
\end{align*}
This last expression is less than $4\pi^2/3\sqrt{3}$ if
\begin{equation} \label{eq:harmonic1}
b+\frac{20\sqrt{3}}{\pi^2} < \frac{q}{b} .
\end{equation}
Thus strict inequality holds in the theorem if \eqref{eq:harmonic1} is true.

Next consider the eigenfunctions $u_1$ and $u_2$ of the equilateral triangle $E$, as in Section~\ref{equilateral}, and the affine transformation $\tau$ from $E$ to $T$. Transplant the eigenfunctions by $v_1=u_1 \circ \tau^{-1}$ and $v_2 = u_2 \circ \tau^{-1}$, so that $v_1$ and $v_2$ integrate to $0$ over $T$. Let $\gamma \in \R$. The trial functions $v_1$ and $\gamma v_1 + v_2$ have Rayleigh quotients
\begin{align*}
  R[v_1] & = \frac{(32\pi^2+243)q-1458}{144b^2}, \\
  R[\gamma v_1 + v_2] & = \frac{[(32\pi^2 + 243)q-1458]\gamma^2 - 972 \sqrt{3} a \gamma + [(32\pi^2 - 243)q+1458]}{144b^2 (\gamma^2 + 1)} ,
\end{align*}
as shown by formula \eqref{eq:ray2} with $\delta=0$ and $\delta=1$, respectively.

We choose the coefficient $\gamma$ so that the gradients of $v_1$ and $\gamma v_1 + v_2$ are orthogonal:
\[
  \gamma=-\frac{\int_T \nabla v_1 \cdot \nabla v_2 \, dA}{\int_T |\nabla v_1|^2 \, dA} =
  \frac{486\sqrt{3}a}{(32\pi^2+243)q-1458}
\]
by formula \eqref{eq:ray3}. Then by Poincar\'{e}'s principle,
\begin{align*}
H(\mu_1,\mu_2)A 
& \le H(R[v_1],R[\gamma v_1 + v_2])A \\
& = \frac{(1024\pi^4-243^2)q^2+12 \cdot 243^2 b^2}{4608\pi^2 bq} ,
\end{align*}
which is less than $4\pi^2/3\sqrt{3}$ if and only if
\begin{equation} \label{eq:harmonic2}
2\sqrt{3} < \frac{q}{b} < \frac{2\sqrt{3} \cdot 243^2}{1024\pi^4 - 243^2} \simeq 5.03 .
\end{equation}
Hence when \eqref{eq:harmonic2} is true, strict inequality holds in the theorem.

To complete the proof, we divide into three cases. First, if $q/b=2\sqrt{3}$ then $T$ is equilateral (because $a=0$ and $b=\sqrt{3}$ by considering equality in \eqref{eq:qb}), so that equality holds in the theorem by Section~\ref{equilateral}. Second, if
\[
2\sqrt{3} < \frac{q}{b} < 5
\]
then strict inequality holds in the theorem by \eqref{eq:harmonic2}. Third, suppose
\[
5 \leq \frac{q}{b} ,
\]
which means $5b \leq a^2+b^2+3$. Since $a<1$ we deduce $0<(b-1)(b-4)$, so that $b<1$. Therefore estimate \eqref{eq:harmonic1} is true, because its left side is at most $1+20\sqrt{3}/\pi^2 \simeq 4.51$ while its right side is at least $5$. Hence once again the theorem holds with strict inequality.

\section{\bf Proof of Theorem \ref{12upAS}}
The arithmetic mean of the first two nonzero eigenvalues is characterized in terms of the Rayleigh quotient by Poincare's variational principle \cite[p.~98]{B80}:
\[
   M(\mu_1,\mu_2) = \min \big\{ M(R[f_1],R[f_2]) : \text{$f_1 \perp 1, f_2 \perp 1$ and $f_1 \perp f_2$} \big\}.
\]

Like in the earlier proofs, we need only consider the triangle $T$ with vertices
$(-1,0), (1,0)$ and $(a,b)$, where $a \geq 0, b>0$ and
\begin{equation} \label{sideconstr}
(a+1)^2+b^2 \leq 2^2 .
\end{equation}
Recall the eigenfunctions $u_1$ and $u_2$ of the equilateral triangle $E$, as in Section~\ref{equilateral}, and the affine transformation $\tau$ from $E$ to $T$. Transplant the eigenfunctions by $v_1=u_1 \circ \tau^{-1}$ and $v_2 = u_2 \circ \tau^{-1}$, so that the trial functions $v_1$ and $v_2$ integrate to $0$ over $T$. Note $\int_T v_1 v_2 \, dA = 0$ by the antisymmetry and symmetry properties of $u_1$ and $u_2$.

The Rayleigh quotients evaluate to
\[
  R[v_1] = \frac{(32\pi^2+243)q-1458}{144b^2} \quad \text{and} \quad
  R[v_2] = \frac{(32\pi^2 - 243)q+1458}{144b^2}
\]
by formula \eqref{eq:ray2} with $\delta=0$ and $\gamma=0$, respectively. Hence
\begin{align*}
M(\mu_1,\mu_2) & \le M(R[v_1],R[v_2]) \\
& = \frac{2\pi^2 q}{9b^2} \\
& =\frac{\pi^2 S^2}{9A^2}
\end{align*}
since $S^2=2q$ by formula \eqref{Sis}, and $A=\frac{1}{2}\cdot 2 \cdot b = b$. This last estimate is the desired inequality.

If the triangle is equilateral then equality holds in the theorem, by Section~\ref{equilateral}.

Suppose equality holds in the theorem. We will show $T$ is equilateral. Since equality holds in our argument above, the arithmetic mean of $\mu_1$ and $\mu_2$ equals $M(R[v_1],R[v_2])$, which implies by the proof of the variational principle (see \cite[p.~98]{B80}) that the span of $v_1$ and $v_2$ equals the span of some two eigenfunctions with eigenvalues $\mu_1$ and $\mu_2$; these eigenfunctions can be assumed orthogonal in $L^2(T)$. Hence there exists a linear combination $v=\gamma v_1 + \delta v_2$ (with coefficients $\gamma,\delta$, not both zero) that is an eigenfunction of the Laplacian on $T$. By direct calculation,
\begin{align*}
\frac{\Delta v}{v} & = \ \ \frac{2\pi^2}{9b^2} (a^2 + b^2 - 9) \qquad \text{at $(x,y)=(0,0)$,} \\
\frac{\Delta v}{v} & = -\frac{2\pi^2}{9b^2} (a^2 + b^2 + 3) \qquad \text{at $(x,y)=(a,b)$,}
\end{align*}
where we used the definition $v=(\gamma u_1 + \delta u_2) \circ \tau^{-1}$ and called on the formulas for $u_1, u_2$ and $\tau^{-1}$ in Section~\ref{equilateral}. These expressions for $(\Delta v)/v$ at $(0,0)$ and at $(a,b)$ must be equal, since $(\Delta v)/v$ is constant by the eigenfunction property. Hence $a^2+b^2=3$. The constraint \eqref{sideconstr} then implies $a \leq 0$, so that $a=0$ and hence $b = \sqrt{3}$. Thus $T$ is equilateral, completing the proof of the equality statement.

\section{\bf Discussion of isoperimetric excess, and open problems} \label{discussion}

One could attempt to strengthen Szeg\H{o} and Weinberger's result \eqref{SW1} for general plane domains by adding a multiple of the general \emph{isoperimetric excess}, which is defined by
\[
  {\mathcal E}=\frac{L^2}{4\pi}-A .
\]
Note the excess is nonnegative by the isoperimetric inequality, and that it
equals $0$ only for disks.

\begin{problem} \label{SWexcess}
Does there exist $\delta>0$ such that
\[
\mu_1 \cdot (A+\delta{\mathcal E}) \leq \pi (j_{1,1}^\prime)^2
\]
for all convex plane domains? What is the largest possible
$\delta$?
\end{problem}

Among general domains,
\[
\text{$\mu_1 L^2$ is \emph{not} maximal for the disk,}
\]
because the equilateral triangle and the square both have $\mu_1 L^2 = 16\pi^2 \simeq
158$, which exceeds the value $4\pi^2 (j_{1,1}^\prime)^2 \simeq
133$ for the disk. Hence Problem~\ref{SWexcess} needs $\delta<1$, because when $\delta=1$ one has $A+\delta {\mathcal E} = L^2/4\pi$. 

\begin{problem} \label{muperimeter}
Determine the maximizers for $\mu_1 L^2$, among all bounded convex
domains in the plane.
\end{problem}
The convexity hypothesis eliminates domains with fractal boundary, for which $L$ is infinite and $\mu_1$ can be positive \cite{P08}. A result somewhat similar to Problem~\ref{SWexcess} was proved by Nadirashvili \cite{N97}, but with a measure theoretic ``asymmetry'' correction instead of the isoperimetric excess.

\medskip
The triangular version of Problem~\ref{SWexcess} is to find $\delta>0$ such that
\begin{equation} \label{excessmu}
   \mu_1 \cdot (A+\delta{\mathcal E}_T) \leq \frac{4\pi^2}{3\sqrt{3}} ,
\end{equation}
where we recall the triangular isoperimetric excess ${\mathcal E}_T = (L^2/12\sqrt{3})-A$ defined in \eqref{eq:triexcess}.

We have already proved triangular excess bounds of the form \eqref{excessmu}: the perimeter bound $\mu_1 L^2 \leq 16 \pi^2$ in Theorem~\ref{1upS} has that form for $\delta=1$, because $A + {\mathcal E}_T =  L^2/12\sqrt{3}$. Theorem~\ref{1opt} is even
stronger, for it proves \eqref{excessmu} with $\delta =
\pi^2/j_{0,1}^2 \simeq 1.7$ and hence with all smaller values of
$\delta$ too, such as $\delta=1$ and $\delta=3/2$.

Theorem~\ref{1opt} implies Theorem~\ref{1upS}, because \eqref{excessmu} with $\delta = 3/2$ implies the
sum-of-squares bound $\mu_1 S^2 \leq 16 \pi^2/3$ by Lemma~\ref{le:twotheorems} below. 
\begin{lemma} \label{le:twotheorems}
For all triangles,
\[
  S^2 \leq \frac{12}{\sqrt{3}} \Big( A+\frac{3}{2}{\mathcal E}_T \Big) ,
\]
with equality for equilateral triangles and asymptotic equality for
degenerate acute isosceles triangles.
\end{lemma}
\begin{proof}
In the notation of Section~\ref{pr:1opt} we have
\begin{align*}
L & = 2(1+r) , \\
A & = \sqrt{(r^2-1)(1-s^2)} , \\
S^2 & = 2(2+r^2+s^2) ,
\end{align*}
where the parameters satisfy $1 < r \leq 2 , 0 \leq s < 1$. Substituting these quantities into the lemma, we see the task is to prove
\[
2r - (1+s^2) \geq \sqrt{3(r^2-1) (1-s^2)} .
\]
The left side is positive. By squaring both sides and rearranging, we reduce to the equivalent inequality
\[
(1+3s^2) \Big( r - 2 \frac{1+s^2}{1+3s^2} \Big)^{\! 2} + 3s^2 \frac{(1-s^2)^2}{1+3s^2} \geq 0 ,
\]
which is clearly true. Equality holds when $r=2,s=0$, which is the equilateral case. Equality holds asymptotically when $s=1,r=1$, which corresponds to a degenerate acute isosceles triangle.
\end{proof}
Incidentally, we settled on the choice of $\delta$ in
Theorem~\ref{1opt} by increasing $\delta$ until some
non-equilateral triangle also gave equality in the theorem. Any
further increase would prevent the equilateral
triangle from being optimal.

\medskip
Is there a \emph{lower} excess bound for $\mu_1$, complementing the
upper bounds in Theorem \ref{1opt} and Problem~\ref{SWexcess}?
\begin{problem} \label{mulower}
Is there a constant $\delta>0$ such that for all triangles,
\[
    \mu_1 (A+\delta {\mathcal E}_T)\ge \frac{4\pi^2}{3\sqrt{3}} \, ?
\]
Is there a constant $\delta>0$ such that for all bounded Lipschitz plane domains,
\[
    \mu_1 (A+\delta {\mathcal E})\ge \pi (j_{1,1}^\prime)^2 \, ?
\]
\end{problem}
For triangles one would need $\delta \geq 4\pi^2/j_{1,1}^2$ at
least, in order for the inequality in the Problem to hold for the degenerate
acute isosceles triangle (see \cite[Corollary~3.5]{LS09b}).
For general domains one would need $\delta \geq ( j_{1,1}^\prime)^2$ at least, in
order for the inequality to hold for the degenerate rectangle.

\medskip
Turning now from the fundamental tone to the harmonic mean of the first two eigenvalues, we raise:
\begin{conjecture}\label{12upL}
For all triangles,
\[
H(\mu_1,\mu_2)L^2 \leq 16\pi^2
\]
with equality if and only if the triangle is equilateral.
\end{conjecture}
This conjecture would be stronger than Theorem~\ref{12upA}, where we used $A$ instead of $L^2$.

For the geometric mean our numerical work similarly suggests:
\begin{conjecture}\label{harmonicA}
For all triangles,
\[
\mu_1 \mu_2 A^2 \leq \frac{16\pi^4}{27}
\]
with equality if and only if the triangle is equilateral.
\end{conjecture}
This conjecture would be stronger than Corollary~\ref{cor:geom},
where we had $A^3/S^2$ instead of $A^2$, in view of
Lemma~\ref{averages}. It would also be stronger than
Theorem~\ref{12upA}, which uses  the harmonic mean and $A$.

To contrast the last two conjectures, notice that when the
eigenvalue functional is strengthened from the harmonic mean to
the geometric mean, the scaling factor is weakened from the
perimeter to the area.

For the geometric mean on general plane domains, Iosif Polterovich has
conjectured $\mu_1 \mu_2 A^2 \leq \pi^2
(j_{1,1}^\prime)^4$ with equality for the disk (see Mathematisches Forschungsinstitut Oberwolfach MFO Report 6/2009). This inequality is
known up to a factor of $2$, by combining the
Szeg\H{o}--Weinberger inequality $\mu_1 A \leq \pi
(j_{1,1}^\prime)^2$ with the recent result of Girouard \emph{et
al.} that $\mu_2 A \leq 2\pi (j_{1,1}^\prime)^2$ (with equality
holding for a domain degenerating suitably to two disjoint disks
of equal area) \cite{GNP08}. Incidentally, the better bound $\mu_2 A \leq \pi (j_{1,1}^\prime)^2$ holds for domains with $4$-fold rotational symmetry \cite{AB93}.

Lastly, for curved surfaces we raise the open problem of maximizing the Neumann fundamental tone among spherical and hyperbolic triangles of given area, in the two dimensional sphere and hyperbolic disk respectively. Note the Szeg\H{o}--Weinberger inequality for general domains has been extended from euclidean space to curved surfaces in two dimensions \cite[{\S}III.3.3]{B80}, and to the sphere and hyperbolic space in all dimensions \cite{AB95,C84}.

\section{\bf Survey of Dirichlet eigenvalue estimates} \label{dirich}

We close the paper by mentioning analogous results for
Dirichlet eigenvalues. The Dirichlet analogue of the
Szeg\H{o}-Weinberger bound \eqref{SW1} is the Faber-Krahn
inequality
\[
  \lambda_1 A\ge \pi j_{0,1}^2 ,
\]
which holds with equality for the disk. The triangular version of
this inequality appears in the book of P\'olya and Szeg\H{o}
\cite[p.~158]{PS51}:
\[
  \lambda_1 A \ge \frac{4\pi^2}{\sqrt{3}}
\]
with equality for equilateral triangles.

Faber--Krahn type bounds are necessarily one-sided, because a
long, thin domain can have fixed area and $\lambda_1$ arbitrarily
large. To obtain a two-sided bound on $\lambda_1$ one must weaken
the geometric functional. For example, for convex plane domains
one has
\[
  \frac{\pi^2}{16} \leq \lambda_1 \frac{A^2}{L^2} \leq \frac{\pi^2}{4} ,
\]
where the upper bound is due to P\'olya \cite{P60} and the lower
bound to Makai \cite{M62}. Equality holds asymptotically in these
bounds for degenerate circular sectors and degenerate rectangles,
respectively. For triangles a sharper upper bound was proved by
Siudeja \cite{S07}:
\begin{equation}\label{dirtri}
  \frac{\pi^2}{16} \leq \lambda_1 \frac{A^2}{L^2} \leq \frac{\pi^2}{9}
\end{equation}
with equality in the upper bound for equilateral triangles.

These last bounds can be strengthened to include the isoperimetric
excess; see Siudeja \cite[Conjecture 1.2]{S07} and Freitas and
Antunes \cite{AF06}.

The geometric functional $A^2/S^2$ that we combined in Theorem \ref{12upAS} with the arithmetic mean of
the Neumann eigenvalues has been studied also in the Dirichlet case. Freitas \cite{F06} showed for
arbitrary triangles that
\[
  \lambda_1 \frac{A^2}{S^2} \le \frac{\pi^2}{3} ,
\]
which is slightly weaker than \eqref{dirtri}; quadrilaterals have been studied too \cite{fresiu}.
%For quadrilaterals one has (by Freitas and Siudeja \cite{fresiu}) that
%\[
%  \lambda_1 \frac{A^2}{S^2} \le \frac{\pi^2}{2} ,
%\]
%which was shown first for rhombi (Hooker and Protter \cite{HP61})
%and parallelograms (Hersch \cite{H66}). 
Conjectures involving $\lambda_1$ and geometric functionals have been raised
by Antunes and Freitas \cite{AF06}.

The Dirichlet gap conjecture for triangles, due to Antunes and Freitas \cite[Conjecture~4]{AF08}, claims that $(\lambda_2 - \lambda_1)D^2$ is minimal for the equilateral triangle. Some progress has been made recently by Lu and Rowlett \cite{LR08}.

Finally, recall the inverse problem of determining the shape
of a triangle from knowledge of its Dirichlet spectrum. The
spectrum is tremendously redundant, since it is determined by
merely three parameters (the side lengths of the triangle). It is
plausible that the triangle could be determined (up to congruence)
by knowing just finitely many eigenvalues. Chang and DeTurck
\cite{CD89} did so nonconstructively, with the required number of
eigenvalues depending on $\lambda_1$ and $\lambda_2$. A
constructive approach or explicit formula for solving the
inverse problem would be most welcome.

\section*{Acknowledgments} We are grateful to Mark Ashbaugh and Iosif Polterovich, who shared their insights on Neumann eigenvalues and provided a number of helpful references.

\end{document}